\documentclass[AMA,STIX1COL]{WileyNJD-v2}
\usepackage{moreverb}

\newcommand\BibTeX{{\rmfamily B\kern-.05em \textsc{i\kern-.025em b}\kern-.08em
T\kern-.1667em\lower.7ex\hbox{E}\kern-.125emX}}

\articletype{Article Type}%

\received{<day> <Month>, <year>}
\revised{<day> <Month>, <year>}
\accepted{<day> <Month>, <year>}
 
\newcommand{\ds}{\displaystyle}
\newcommand{\df}{\displaystyle\frac}

\newcommand{\om}{\omega}

\begin{document}

\title{Stability and Hopf bifurcation analysis of a four-dimensional hypothalamic-pituitary-adrenal axis model with distributed delays \protect\thanks{This work was supported by a grant of the Romanian National Authority for
Scientific Research and Innovation, CNCS-UEFISCDI, project no. PN-II-RU-TE-2014-4-0270.}}

\author[1,2]{Eva Kaslik}
\author[1]{Mihaela Neam\c{t}u*}

\authormark{EVA KASLIK \& MIHAELA NEAM\c{T}U}

\address[1]{\orgname{West University of Timi\c{s}oara}, \country{Romania}}

\address[2]{\orgname{Institute e-Austria Timi\c{s}oara}, \country{Romania}}

\corres{*Mihaela Neam\c{t}u, \email{mihaela.neamtu@e-uvt.ro}}


\abstract[Abstract]{
A four-dimensional mathematical model of the hypothalamus-pituitary-adrenal (HPA) axis is investigated, incorporating the influence of the GR concentration and general feedback functions. The inclusion of distributed time delays provides a more realistic modeling approach, since the whole past history of the variables is taken into account. The positivity of the solutions and the existence of a positively invariant bounded region are proved. It is shown that the considered four-dimensional system has at least one equilibrium state and a detailed local stability and Hopf bifurcation analysis is given. Numerical results reveal the fact that an appropriate choice of the system's parameters leads to the coexistence of two asymptotically stable equilibria in the non-delayed case. When the total average time delay of the system is large enough, the coexistence of two stable limit cycles is revealed, which successfully model the ultradian rhythm of the HPA axis both in a normal disease-free situation and in a diseased hypocortisolim state, respectively. Numerical simulations reflect the importance of the theoretical results.} 



\maketitle

\section{Introduction}

The hypothalamus-pituitary-adrenal (HPA) axis is a neuroendocrine system which regulates a number of physiological processes \cite{Conrad_2009,kim2016onset}, playing an important role in stress response. It consists of the hypothalamus, pituitary and adrenal glands, as well direct influences and positive and negative feedback interactions. Different types of stressors (e.g. infection, dehydration, anticipation, fear) activate the secretion of corticotropin-releasing hormone (CRH)  in the hypothalamus, which induces the corticotropin (ACTH) production in the pituitary. ACTH travels by the bloodstream to the adrenal cortex, where it activates the release of cortisol (CORT), which in turn down-regulates the production of both CRH and ACTH.

Dynamical systems have previously proved to be successful in studying metabolic and endocrine processes. Different types of mathematical models of the HPA axis have been recently explored. Three dimensional systems of differential equations with or without time delays, with the state variables given by the hormone concentrations CRH, ACTH and CORT, have been used to model the HPA axis in \cite{Jelic_2005,Lenbury_2005,Savic_2006,Bairagi_2008,Pornsawad_2013}. The influence of the circadian rhythm in the mathematical model has been analyzed in \cite{bangsgaard2017}. A more general three-dimensional model has been developed in \cite{Vinther_2011}, possessing a unique equilibrium state. If time delays are not taken into consideration, no oscillatory behavior has been observed \cite{Vinther_2011, Andersen_2013}. Oscillatory solutions should be a feature of mathematical models of the HPA axis, as they correspond to the circadian / ultradian rhythm of hormone levels \cite{Carroll_2007}. A generalization of the "minimal model" \cite{Vinther_2011}  has been obtained in \cite{Kaslik_Neamtu_2016}, including memory terms in the form of distributed delays and fractional-order derivatives, which are shown to generate oscillatory solutions. 

Due to the transportation of the hormones throughout the HPA axis, time delays should mandatorily be incorporated in the considered mathematical models. With the aim of reflecting the whole past history of the variables, general distributed delays are considered, proving to be more realistic and more accurate in real world applications than discrete time delays \cite{Cushing_2013}. Distributed delay models appear in a wide range of applications such as hematopoiesis \cite{Adimy_2006}, population biology \cite{Faria_2008,song2017dd,feng2017dd} or neural networks \cite{Jessop_2010,du2013dd}.

Four-dimensional models which incorporate the positive self-regulation of glucocorticoid receptors (GR) in the pituitary have been investigated in \cite{Gupta_2007,benzvi2009,Sriram_2012,zarzer2013,CMMSE2017}. In particular, in \cite{CMMSE2017} we constructed a four-dimensional general model with distributed time delays, which represents an extension of the minimal model of \cite{Vinther_2011}. In \cite{Gupta_2007}, it has been suggested that positive self-regulation of GR may trigger bistability in the dynamical structure of the HPA model, i.e. there exist two asymptotically stable equilibrium states: one corresponding to the normal disease-free state with higher cortisol levels, and a second one with lower cortisol levels related to a diseased state associated with hypocortisolism. 

In this paper, an in-depth analysis is provided for the distributed-delay model introduced in \cite{CMMSE2017}, proving the positivity of the solutions and the existence of a positively invariant bounded region. It is shown that the considered four-dimensional system has at least one equilibrium state and a local stability and bifurcation analysis is provided. Numerical results reveal the fact that an appropriate choice of the system's parameters leads to the coexistence of two asymptotically stable equilibria in the non-delayed case. Moreover, when the total average time delay is large enough, it is shown that two stable limit cycles coexist, which appear due to Hopf bifurcations, extending the results presented in \cite{Gupta_2007,CMMSE2017}. 

\section{Mathematical model of HPA with distributed delays}

With the aim of formulating a mathematical model of the HPA axis, the following sequence of events is considered. Cognitive and physical stressors stimulate CRH neurons in the paraventricular nucleus
(PVN) of the hypothalamus to trigger the secretion of
corticotropin-releasing hormone (CRH), which is released into the portal blood vessel of the hypophyseal stalk. CRH is transported to the anterior pituitary, where it stimulates the secretion of adrenocorticotropin
hormone (ACTH), with an average time delay $\tau_1$. ACTH then activates a complex signaling cascade in the adrenal
cortex, stimulating the secretion of the stress hormone cortisol (CORT) with the average time delay $\tau_2$. CORT exerts a negative feedback on the hypothalamus and the pituitary, suppressing the synthesis and release of CRH and ACTH, in an effort to return them to the baseline levels. On one hand, cortisol inhibits the secretion of CRH in the hypothalamus \cite{Landsberg_1992}, with an average time delay $\tau_{31}$. On the other hand, CORT binds to glucocorticoid receptors (GR) in the pituitary and performs a negative feedback on the secretion of ACTH, with an average time delay $\tau_{32}$. Moreover, the CORT-GR complex self-upregulates the GR production in the anterior pituitary, with an average time delay $\tau_{34}$ .

Denoting the plasma concentrations of hormones CRH, ACTH and CORT by $x_1(t)$, $x_2(t)$, and $x_3(t)$ respectively, and the availability of the glucocorticoid receptor GR in the anterior pituitary  by $x_4(t)$, the following system of differential equations with general distributed delays is considered:
\begin{equation}\label{sys.hpa.gr.dd}
\left\{\begin{array}{l}
\ds\dot x_1(t)=k_1f_1\left(\int_{-\infty}^t x_3(s)h_{31}(t-s)ds\right)-w_1x_1(t),\\
\ds \dot x_2(t)=k_2f_2\left(x_4(t)\int_{-\infty}^t x_3(s)h_{32}(t-s)ds\right)\int_{-\infty}^tx_1(s)h_1(t-s)ds-w_2x_2(t),\\
\ds \dot  x_3(t)=k_3\int_{-\infty}^tx_2(s)h_2(t-s)ds-w_3x_3(t),\\
\ds \dot x_4(t)=k_4\left(\xi+f_3\left(x_4(t)\int_{-\infty}^t x_3(s)h_{34}(t-s)ds\right)\right)-w_4x_4(t).\\
\end{array}\right.
\end{equation}
Here, the positive constants $k_i$, $i=\overline{1,4}$, relate the production rate of each variable to specific factors that regulate the rate of release/synthesis \cite{kim2016onset}. The basal production rate $\xi$ and elimination constants $w_1,w_2,w_3,w_4$ are positive.


The function $f_1$ represents the negative feedback of CORT on CRH levels in the paraventricular nucleus of the hypothalamus while the function $f_2$ describes the negative feedback of the CORT-GR complex (at concentration $x_3(t)x_4(t)$) in the pituitary. The positive feedback function $f_3$, describes the self-upregulation effect of the CORT-GR complex on GR production in the anterior pituitary. The following general assumptions will be considered: 
\begin{itemize}
\item $f_1,f_2:[0,\infty)\to (0,1]$ are strictly decreasing, smooth and bounded on $[0,\infty)$;
\item $f_3:[0,\infty)\to [0,1)$ is strictly increasing, smooth and bounded on $[0,\infty)$;
\item $f_1(0)=f_2(0)=1$; $f_3(0)=0$.
\end{itemize}
As a special case, the feedback functions can be chosen as Hill functions, such as in \cite{kim2016onset,Vinther_2011,Andersen_2013,Gupta_2007,Sriram_2012}, which verify the conditions given above:
\begin{equation}\label{func.hill.neg}
f_1(u)=1-\eta\df{u^{\alpha_1}}{c_1^{\alpha_1}+u^{\alpha_1}}\quad,\quad 
f_2(u)=1-\mu\df{u^{\alpha_2}}{c_2^{\alpha_2}+u^{\alpha_2}}\quad,\quad f_3(u)=\df{u^{\alpha_3}}{c_3^{\alpha_3}+u^{\alpha_3}}
\end{equation}
with Hill coefficients $\alpha_1,\alpha_2,\alpha_3\geq 1$, $\eta,\mu\in(0,1]$, and microscopic dissociation constants $c_1,c_2,c_3>0$. 

In system (\ref{sys.hpa.gr.dd}), the delay kernels $h_1,h_2,h_{31},h_{32},h_{34}:[0,\infty)\to[0,\infty)$ are probability density functions representing the probability of occurrence of a particular time delay. These functions are bounded, piecewise continuous and satisfy
\begin{equation}\label{delay.kernel.properties}
\int_0^{\infty}h(s)ds=1.
\end{equation}
The average time delay of a kernel $h(t)$ is
$$\tau=\int_0^{\infty}sh(s)ds<\infty.$$
In this paper, we focus our attention on two types of delay kernels:
\begin{itemize}
\item Dirac kernels: $h(s)=\delta(s-\tau)$, where $\tau\geq 0$, equivalent to a discrete time delay:
$$\int_{-\infty}^t x(s)h(t-s)ds=\int_0^\infty x(t-s)\delta(s-\tau)ds=x(t-\tau).$$
\item Gamma kernels: $h(s)=\df{ s^{p-1}e^{-s/\theta}}{\theta^p\Gamma(p)}$, where $p,\theta>0$, with the average delay $\tau=p\theta$.
\end{itemize}
In the mathematical modeling of real world phenomena, the exact distribution of time delays is generally unavailable, and hence, general kernels may provide better results \cite{Campbell_2009,Yuan_2011}. The analysis of models which include particular classes of delay kernels (e.g. weak Gamma kernels with $p=1$ or strong Gamma kernels with $p=2$) may reveal the more realistic effect of distributed delays on the system's dynamics, compared to discrete delays. 

Initial conditions associated with system (\ref{sys.hpa.gr.dd}) are of the form:
$$x_i(s)=\varphi_i(s),\quad \forall\, s\in(-\infty,0], \,\,i=1,2,3,4,$$
where $\varphi_i$  are bounded continuous functions defined on $(-\infty,0]$, with values in $[0,\infty)$.

\section{Positively invariant sets and equilibrium states}

\begin{lemma}\label{lem.1} Assume that $g:[0,\infty)\rightarrow[0,\infty)$ is a continuously differentiable function such that there exist $m_1,m_2>0$  such that $\ds g(0)\leq \frac{m_1}{m_2}$ and
$$g'(t)\leq m_1-m_2g(t),\quad\forall~t\geq 0.$$
Then, $\ds g(t)\leq \frac{m_1}{m_2}$ for any $t\geq 0$.
\end{lemma}

\begin{proof}
From the hypothesis we easily obtain that the function $G(t)=e^{m_2t}\left(g(t)-\frac{m_1}{m_2}\right)$ is decreasing on $[0,\infty)$. 
Therefore, as $G(t)\leq G(0)$ for any $t\geq 0$, it follows that
$$g(t)\leq \frac{m_1}{m_2}+e^{-m_2t}\left(g(0)-\frac{m_1}{m_2}\right)\leq \frac{m_1}{m_2},\quad\forall~t\geq 0.$$
This completes the proof.
\end{proof}

In the following, we denote: 
$$\df{k_1}{w_1}=L_1\quad,\quad \df{k_1k_2}{w_1w_2}=L_2\quad,\quad \df{k_1k_2k_3}{w_1w_2w_3}=L_3\quad,\quad\df{k_4}{w_4}=L_4.$$

\begin{proposition}\label{prop.dissipativity}
The compact set
$$\Omega=\left[0,L_1\right]\times\left[0,L_2\right]\times\left[0,L_3\right]\times\left[0,(\xi+1)L_4\right]\subset \mathbb{R}^4_+$$
and $\mathbb{R}^4_+$ are positively invariant sets for system (\ref{sys.hpa.gr.dd}). 

\end{proposition}

\begin{proof}
Assume that $(x_1(t),x_2(t),x_3(t),x_4(t))$ denotes the solution of system (\ref{sys.hpa.gr.dd}) with the initial condition $x_i(s)=\varphi_i(s)$, $s\in(-\infty,0]$, with $i=\overline{1,4}$, where $\varphi_i$ are bounded positive continuous functions defined on $(-\infty,0]$. From the positivity of the feedback functions it easily follows that
$$\dot{x}_i(t)\geq -w_i x_i(t),\quad\forall~t>0,~i=\overline{1,4}$$
and hence, the functions $x_i(t)e^{w_it}$ are increasing on $(0,\infty)$. Therefore:
$$x_i(t)\geq \varphi_i(0)e^{-w_it}\geq 0,\quad\forall~t>0,~i=\overline{1,4}.$$
Therefore, all positive initial conditions lead to positive solutions, i.e.  $\mathbb{R}^4_+$ is positively invariant for system (\ref{sys.hpa.gr.dd}).

Moreover, assume $(\varphi_1(s),\varphi_2(s),\varphi_3(s),\varphi_4(s))\in\Omega$ for any $s\in(-\infty,0]$.

From the first equation of  (\ref{sys.hpa.gr.dd}) and the boundedness of $f_1$, it follows that 
$$\dot{x}_1(t)\leq k_1-w_1x_1(t),\quad\forall~t>0.$$
Using Lemma \ref{lem.1}, as $x_1(0)\leq L_1$, we have that $x_1(t)\leq L_1$ for any $t\geq 0$. 

The second equation of (\ref{sys.hpa.gr.dd}), the boundedness of $f_2$ and (\ref{delay.kernel.properties}) provides 
$$\dot{x_2}(t)\leq k_2L_1-w_2x_2(t),\quad\forall~t> 0.$$
From Lemma \ref{lem.1} it follows that $x_2(t)\leq L_2$ for any $t\geq 0$.

From the third equation of (\ref{sys.hpa.gr.dd}) and (\ref{delay.kernel.properties}) it follows that 
$$\dot{x_3}(t)\leq k_3L_2-w_3x_3(t),\quad\forall~t\geq 0.$$
Lemma \ref{lem.1} leads to $x_3(t)\leq L_3$ for any $t\geq 0$.

The last equation of (\ref{sys.hpa.gr.dd}), the boundedness of $f_3$ leads to
$$\dot{x_4}(t)\leq k_4(\xi+1)-w_4x_4(t),\quad\forall~t\geq 0,$$
which, based on Lemma \ref{lem.1}, provides the desired conclusion.
\end{proof}

\begin{remark}
Due to the fact that $x_4(t)$ in the mathematical model (\ref{sys.hpa.gr.dd}) is a non-dimensional variable  representing the availability of glucocorticoid receptors  \cite{kim2016onset,Gupta_2007}, it is reasonable to demand that $x_4(t)\in[0,1]$ for any $t\in\mathbb{R}$. Based on Proposition \ref{prop.dissipativity}, this is guaranteed if the following inequality is satisfied: 
$$(\xi+1)L_4\leq 1.$$
\end{remark}

The existence of an equilibrium point of system (\ref{sys.hpa.gr.dd}) is provided by the following:

\begin{proposition}\label{prop.eq.states}
The equilibrium states of system (\ref{sys.hpa.gr.dd}) belong to the invariant set $\Omega$ and are of the form
\begin{equation}\label{eq.state}
E=\left(L_1f_1(x_0), \df{w_3x_0}{k_3},x_0,\df{1}{x_0}f_2^{-1}\left(\df{x_0}{L_3f_1(x_0)}\right)\right).
\end{equation}
where $x_0\in\left[0,L_3\right]$ is a solution of the equation
\begin{equation}\label{eq.states}
L_4\left(\xi+(f_3\circ f_2^{-1})\left(\df{x}{L_3f_1(x)}\right)\right)=\df{1}{x}f_2^{-1}\left(\df{x}{L_3f_1(x)}\right).
\end{equation}
\end{proposition}

\begin{proof} From Proposition \ref{prop.dissipativity} it follows that any equilibrium state of system (\ref{sys.hpa.gr.dd}) belongs to the set $\Omega$. Moreover, An equilibrium point of system (\ref{sys.hpa.gr.dd}) is a solution of the following algebraic system:
\begin{equation}
\left\{\begin{array}{l}
k_1f_1(x_3)=w_1x_1,\\
k_2f_2(x_3x_4)x_1=w_2x_2,\\
k_3x_2=w_3x_3,\\
k_4(\xi+f_3(x_3x_4))=w_4x_4,
\end{array}\right.
\end{equation}
which is equivalent to
\begin{equation}\label{sys.eq.points}
\left\{\begin{array}{l}
\ds x_1=L_1f_1(x_3),\\
\ds x_2=\df{w_3x_3}{k_3},\\
\ds L_3f_2(x_3x_4)f_1(x_3)=x_3,\\
\ds L_4(\xi+f_3(x_3x_4))=x_4.
\end{array}\right.
\end{equation}
From the first two equations of (\ref{sys.eq.points}) it follows that the first two components of an equilibrium state are uniquely determined by the third component. The last two components of an equilibrium state represent a fixed point for the continuous function $F:\mathbb{R}^2\rightarrow \mathbb{R}^2$ defined by 
$$(u,v)\mapsto F(u,v)=\left(L_3f_1(u)f_2(uv),L_4(\xi+f_3(uv))\right)$$
From the boundedness properties of the functions $f_i$, $i\in\{1,2,3\}$ it easily follows that the function $F$ maps the convex compact set $[0,L_3]\times[0,(\xi+1)L_4]$ into itself. By Brouwer's fixed-point theorem we obtain the existence of at least one fixed point of the function $F$ in the set $[0,L_3]\times[0,(\xi+1)L_4]$. Therefore, system (\ref{sys.hpa.gr.dd}) has at least one equilibrium state.

From system (\ref{sys.eq.points}) we easily deduce (\ref{eq.states}), and hence we obtain the form of the equilibrium states given by (\ref{eq.state}).
\end{proof}

\begin{remark}
In the case of the minimal model of the HPA-axis, it has been shown \cite{Vinther_2011,Kaslik_Neamtu_2016} that there exists a unique equilibrium state. For the extended four-dimensional model (\ref{sys.hpa.gr.dd}), Proposition \ref{prop.eq.states} only shows the existence of at least one equilibrium state. The presence of the positive feedback function is often associated with the coexistence of several equilibrium states \cite{Gupta_2007,Sriram_2012}. 
\end{remark}

\section{Local stability analysis}

In this section, necessary and sufficient conditions for the local asymptotic stability of an equilibrium point $E$ are provided, choosing general delay kernels. Delay independent sufficient conditions are explored for the local asymptotic stability of the equilibrium point $E$, which may prove to be useful if the time delays in system (\ref{sys.hpa.gr.dd}) cannot be accurately estimated.

By linearizing the system (\ref{sys.hpa.gr.dd}) at an equilibrium point $E$, we obtain:
\begin{equation}\label{sys.lin}
\left\{\begin{array}{l}
\ds \dot y_1(t)=k_1f_1'(x_0)\int_{-\infty}^ty_3(s)h_{31}(t-s)ds-w_1y_1(t),\\
\ds \dot y_2(t)=k_2f_2(x_0 r_0)\int_{-\infty}^ty_1(s)h_1(t-s)ds+\df{k_1k_2}{w_1}f_1(x_0)r_0 f_2'(x_0 r_0)\int_{-\infty}^ty_3(s)h_{32}(t-s)ds+\\
\qquad\quad+\df{k_1k_2}{w_1}f_1(x_0)x_0 f_2'(x_0 r_0)y_4(t)-w_2y_2(t),\\
\ds \dot y_3(t)=k_3\int_{-\infty}^ty_2(s)h_2(t-s)ds-w_3y_3(t),\\
\ds \dot y_4(t)= k_4r_0 f_3'(x_0 r_0)\int_{-\infty}^ty_3(s)h_{34}(t-s)ds+k_4x_0 f_3'(x_0 r_0)y_4(t)-w_4y_4(t).
\end{array}\right.
\end{equation}
where $r_0=\df{1}{x_0}f_2^{-1}\left(\df{x_0}{L_3f_1(x_0)}\right)$.

The characteristic equation of the linearized system at the equilibrium point $E$ is:
\begin{align}\label{eq.char}
&( z+w_1)( z+w_2)( z+w_3)(z+\tilde{w_4})+a(w_4-\tilde{w_4})( z+w_1)H_2( z)H_{34}( z)+\\
\nonumber &+b(z+\tilde{w_4})H_1( z)H_2( z)H_{31}( z)+a(z+w_1)(z+\tilde{w_4})H_2(z)H_{32}(z)=0,
\end{align}
where $H_i( z)=\int_{0}^\infty e^{- z s}h_i(s)ds$ are the Laplace transforms of the kernels $h_i$, $i\in\{1,2,31,32,34\}$ and
\begin{align}
\label{eq.a} a&=-\df{k_1k_2k_3}{w_1}f_1(x_0)f_2'(x_0 r_0)r_0=-w_2w_3\df{x_0r_0f_2'(x_0r_0)}{f_2(x_0r_0)}>0,\\
\label{eq.b} b&=-k_1k_2k_3f_1'(x_0)f_2(x_0 r_0)=-w_1w_2w_3\df{x_0f_1'(x_0)}{f_1(x_0)}>0,\\
\label{eq.c} \tilde{w_4}&=w_4-k_4x_0 f_3'(x_0 r_0)<w_4.
\end{align}

For the theoretical analysis, we introduce the following set of inequalities:
\begin{align*}
(I_0)\qquad & \tilde{w_4}>0;\\
(I_1)\qquad & (w_1+\tilde{w_4})(w_2+\tilde{w_4})(w_3+\tilde{w_4})\geq (\tilde{w_4}-w_1)(\tilde{w_4}-w_4)(w_1+w_2+w_3+\tilde{w_4});\\
(I_2)\qquad & a(w_1+w_4)+b\leq (w_1+w_2)(w_2+w_3)(w_1+w_3);\\
(I_3)\qquad & \df{aw_4}{\tilde{w_4}}+\df{b}{w_1}<w_2w_3; \\
(\overline{I_3})\qquad & \df{aw_4}{\tilde{w_4}}+\df{b}{w_1}\geq w_2w_3.
\end{align*}

\begin{theorem}[Local asymptotic stability]\label{thm.stab} $ $
\begin{enumerate}
\item If there is no time-delay and $(I_0)$, $(I_1)$ and $(I_2)$ are satisfied, the equilibrium point $E$ of system (\ref{sys.hpa.gr.dd}) is locally asymptotically stable.
\item For any delay kernels $h_i(t)$, $i\in\{1,2,31,32,34\}$, if $(I_0)$ and $(I_3)$ hold, then the equilibrium point $E$ of system (\ref{sys.hpa.gr.dd}) is locally asymptotically stable.
\end{enumerate}
\end{theorem}

\begin{proof} 1. In the absence of delays, the characteristic equation (\ref{eq.char}) is given by:
\begin{equation}
z^4+c_1z^3+c_2z^2+c_3z+c_4=0,
\end{equation} where
\begin{align*}
c_1&=w_1+w_2+w_3+\tilde{w_4}>0, \\
c_2&=w_1w_2+w_2w_3+w_1w_3+(w_1+w_2+w_3)\tilde{w_4}+a>0,\\
c_3&=w_1w_2w_3+(w_1w_2+w_2w_3+w_1w_3)\tilde{w_4}+a(w_1+w_4)+b>0,\\
c_4&=(w_1w_2w_3+b)\tilde{w_4}+aw_1w_4>0.
\end{align*}
Based on the Routh-Hurwitz stability test, it suffices to prove that
$$
c_1c_2c_3-c_3^2-c_1^2c_4>0.
$$
From this inequality it clearly follows that $c_1c_2-c_3>0$.

Denoting
\begin{align*}
S&=(w_1+w_2)(w_1+w_3)(w_2+w_3)\\
T&=(w_1+\tilde{w_4})(w_2+\tilde{w_4})(w_3+\tilde{w_4})
\end{align*}
we obtain
\begin{align*}
c_1c_2c_3-c_3^2-c_1^2c_4=&(S-b-a(w_1+w_4))(T+b+a(w_1+w_4))+\\
&+a(w_1+w_2+w_3+\tilde{w_4})(T-(\tilde{w_4}-w_1)(\tilde{w_4}-w_4)(w_1+w_2+w_3+\tilde{w_4})))
\end{align*}
Using inequalities $(I_0)$, $(I_1)$ and $(I_2)$ it is easy to see that $c_1c_2c_3-c_3^2-c_1^2c_4>0$. The Routh-Hurwitz stability criterion implies that the equilibrium point $E$ is asymptotically stable.

2. In the presence of delays, the characteristic equation (\ref{eq.char}) can be expressed as $$\varphi(z)=\psi(z),$$
where $\varphi$ and $\psi$ are 
\begin{align*}
\varphi(z)&=-( z+w_1)( z+w_2)( z+w_3)(z+\tilde{w_4}),\\
\psi(z)&=a(w_4-\tilde{w_4})( z+w_1)H_2( z)H_{34}( z)+b(z+\tilde{w_4})H_1( z)H_2( z)H_{31}( z)+a(z+w_1)(z+\tilde{w_4})H_2(z)H_{32}(z).
\end{align*}
The functions  $\varphi$ and $\psi$ are holomorphic in the right half-plane.

Considering $z\in \mathbb{C}$ with $\Re(z)\geq 0$, the properties of the delay kernels (\ref{delay.kernel.properties}) imply:
\begin{align*}
|H_i(z)|=\left|\int_{0}^\infty e^{-z s}h_i(s)ds\right|\leq\int_0^\infty |e^{-z s}|h_i(s)ds=\int_0^\infty e^{-\Re(z) s}h_i(s)ds\leq \int_0^\infty h_i(s)ds=1,
\end{align*}
for any $i\in\{1,2,31,32,34\}$. Therefore, based on inequalities $(I_0)$ and $(I_3)$, we have:
\begin{align*}
|\psi(z)|&\leq a(w_4-\tilde{w_4})|z+w_1||H_2(z)||H_{34}(z)|+b|z+\tilde{w_4}||H_1(z)||H_2(z)||H_{31}(z)|+\\
&\quad+ a|z+w_1||z+\tilde{w_4}||H_2(z)||H_{32}(z)|\\
 &\leq a(w_4-\tilde{w_4})|z+w_1|+b|z+\tilde{w_4}|+a|z+w_1||z+\tilde{w_4}|\\
&=|z+w_1||z+\tilde{w_4}|\left(\frac{a(w_4-\tilde{w_4})}{|z+\tilde{w_4}|}+\frac{b}{|z+w_1|}+a\right)\\
&\leq|z+w_1||z+\tilde{w_4}|\left(\frac{a(w_4-\tilde{w_4})}{\tilde{w_4}} +\frac{b}{w_1}+a\right)\\
&<|z+w_1||z+\tilde{w_4}| w_2w_3 \\
&=|z+w_1||z+w_2||z+w_3||z+\tilde{w_4}|=|\varphi(z)|.
\end{align*}
where the inequality $|z+w|\geq w$, for any  $z\in \mathbb{C}$ with $\Re(z)\geq 0$ and $w>0$, has been repeatedly used.

Hence, the inequality $|\psi(z)|<|\varphi(z)|$ is true for any $z$ in the right half plane, and Rouch\'{e}'s theorem implies that the characteristic equation (\ref{eq.char}) does not have any root in the right half-plane (or on the imaginary axis). Therefore, all the roots of (\ref{eq.char}) are in the open left half plane, and it follows that the equilibrium $E$ is asymptotically stable.
\end{proof}

\begin{remark}
Assume that  $(I_0)$ holds and that the delay kernels $h_i(t)$, $i\in\{1,2,31,32,34\}$ are chosen. If the equilibrium point $E$ of system (\ref{sys.hpa.gr.dd}) is unstable, Theorem \ref{thm.stab} implies that inequality $(\overline{I_3})$ holds. 
\end{remark}

\section{Bifurcation analysis}
In this section, we explore the possibility of the occurrence of limit cycles in a neighborhood of $E$, due to Hopf bifurcations, that reflect the ultradian rhythm of the HPA axis.

For simplicity, we further assume that
$$
H_{32}(z)=H_{34}(z)=H_1(z)H_{31}(z),
$$
and we denote
$$H(z)=H_2(z)H_{32}(z)=H_2(z)H_{34}(z)=H_1(z)H_2(z)H_{31}(z).$$
We emphasize that $H(z)$ is the Laplace transform of the convolution of $h_2$ and $h_{32}$:
$$h(t)=\int_0^t h_2(s)h_{32}(t-s)ds,$$
with the average time-delay
\begin{equation}\label{eq.tau}
\tau=\int_0^{\infty}sh(s)ds=\tau_2+\tau_{32}=\tau_2+\tau_{34}=\tau_1+\tau_2+\tau_{31},
\end{equation}
where $\tau_i$ represent the average delays of the kernels $h_i$, for any $i\in\{1,2,31,32,34\}$. 

The characteristic equation (\ref{eq.char}) is
$$
( z+w_1)( z+w_2)( z+w_3)(z+\tilde{w_4})+[a(z+w_1)(z+w_4)+b(z+\tilde{w_4})]H(z)=0,
$$
which can be rewritten as:
\begin{equation}\label{eq.char.bif}
H(z)^{-1}=Q(z),
\end{equation}
where
$$Q(z)=-\frac{a(z+w_1)(z+w_4)+b(z+\tilde{w_4})}{( z+w_1)( z+w_2)( z+w_3)(z+\tilde{w_4})}.$$
The properties of the function $Q(z)$ are given in the following Lemma.

\begin{lemma}\label{lem.Q}
Assume that $(I_0)$ holds. 
\begin{itemize}
\item[a.] The function
$$\om\mapsto|Q(i\om)|=\sqrt{\frac{(b\tilde{w_4}+aw_1w_4-a\om^2)^2+\om^2(a(w_1+w_4)+b)^2}{(\om^2+w_1^2)(\om^2+w_2^2)(\om^2+w_3^2)(\om^2+\tilde{w_4}^2)}}$$
defined on $[0,\infty)$ is strictly decreasing.  

\item[b.] A unique  positive real root $\om_0$ exists for the equation $|Q(i\om)|=1$ if and only if inequality $(\overline{I3})$ holds.

\item[c.] The function $Q$ satisfies the following inequality:
$$\Im\left(\df{Q'(i\om)}{Q(i\om)}\right)>0\qquad\forall\,\om>0.$$
\end{itemize}
\end{lemma}

\begin{proof}
To prove, a. it is easy to see that
$$
|Q(i\om)|^2=\frac{1}{(\om^2+w_2^2)(\om^2+w_3^2)}\left[a^2+\frac{d_1}{(\om^2+w_1^2)}+\frac{d_2}{(\om^2+\tilde{w_4}^2)}\right]
$$
where
\begin{align*}
d_1=& b^2+2ab\frac{w_1(w_1+w_4)}{w_1+\tilde{w_4}}>0\\
d_2=& 2ab\frac{\tilde{w_4}(w_4-\tilde{w_4})}{w_1+\tilde{w_4}}>0
\end{align*}
Therefore, $\om\mapsto|Q(i\om)|$ is strictly decreasing on $[0,\infty)$, and tends to $0$ as $\om\rightarrow\infty$. Therefore, the equation $|Q(i\om)|=1$ admits a unique positive solution if and only if $|Q(0)|>1$. This implies  $w_1w_2w_3\tilde{w_4}<aw_1w_4+b\tilde{w_4}$, which in turn, is equivalent to $(\overline{I3})$, and b. is proved.

Point c. follows from \cite{Kaslik_Neamtu_2016}. 
\end{proof}

For the bifurcation analysis, due to the complexity of the problem, we restrict our attention to Dirac kernels and Gamma kernels.

\subsection{Dirac kernels}

If all delay kernels are of Dirac type: $h_1(t)=\delta(t-\tau_1)$, $h_2(t)=\delta(t-\tau_2)$, $h_{31}(t)=\delta(t-\tau_{31})$, $h_{32}(t)=\delta(t-\tau_{32})$, $h_{34}(t)=\delta(t-\tau_{34})$ where $\tau_1,\tau_2,\tau_{31},\tau_{32},\tau_{34}\geq 0$ satisfy the property
\begin{equation}\label{cond.dirac}
\tau_2+\tau_{32}=\tau_2+\tau_{34}=\tau_1+\tau_2+\tau_{31}=\tau>0,
\end{equation}
then, the characteristic equation (\ref{eq.char.bif}) becomes:
\begin{equation}\label{eq.char.Dirac}
e^{\tau z}=Q(z).
\end{equation}
Choosing $\tau$ as bifurcation parameter and following the same proof as in \cite{Kaslik_Neamtu_2016}, we have:

\begin{theorem}[Hopf bifurcations; Dirac kernels]\label{thm.bif.dirac}
If inequalities  $(I_0)$, $(I_1)$, $(I_2)$ and $(\overline{I_3})$  hold, considering $\om_0>0$ given by Lemma \ref{lem.Q} and
\begin{equation}\label{eq.tau.Dirac}
\tau_p=\frac{\arccos\left[\Re(Q(i\om_0))\right]+2p\pi}{\om_0},\quad p\in\mathbb{Z}^+,
\end{equation}
the equilibrium point $E$ is asymptotically stable if any only if $\tau\in[0,\tau_0)$. For any $p\in\mathbb{Z}^+$, at $\tau=\tau_p$, a Hopf bifurcation takes place in a neighborhood of the equilibrium point $E$ of system (\ref{sys.hpa.gr.dd}).
\end{theorem}

\subsection{Gamma kernels}

If all delay kernels are of Gamma type: $h_1(t)=\df{t^{p_1-1}e^{-t/\theta}}{\theta^{p_1}(p_1-1)!}$, $h_2(t)=\df{t^{p_2-1}e^{-t/\theta}}{\theta^{p_2}(p_2-1)!}$, $h_{31}(t)=\df{t^{p_{31}-1}e^{- t/\theta}}{\theta^{p_{31}} (p_{31}-1)!}$,  $h_{32}(t)=\df{ t^{p_{32}-1}e^{-t/\theta}}{\theta^{p_{32}}(p_{32}-1)!}$,  $h_{34}(t)=\df{ t^{p_{34}-1}e^{-t/\theta}}{\theta^{p_{34}}(p_{34}-1)!}$, where $\theta>0$ and $p_1,p_2,p_{31},p_{32},p_{34}\in\mathbb{Z}^+\setminus\{0\}$ satisfy:
$$p_2+p_{32}=p_2+p_{34}=p_1+p_2+p_{31}=p\geq 2,$$
the characteristic equation (\ref{eq.char}) is:
\begin{equation}\label{eq.char.Gamma}
(\theta z+1)^p=Q(z).
\end{equation}
Choosing $\theta$ as bifurcation parameter, as in \cite{Kaslik_Neamtu_2016}, the following result holds:

\begin{theorem}[Hopf bifurcations; Gamma kernels]\label{thm.bif.gamma}
If inequalities $(I_0)$, $(I_1)$, $(I_2)$ and $(\overline{I_3})$ hold and $\om_p$ is the largest real root from the interval $(0,\om_0)$ of the equation
\begin{equation}\label{eq.omega.Gamma}
T_p\left(\df{1}{|Q(i\om)|^{1/p}}\right)=\frac{\Re(Q(i\om))}{|Q(i\om)|}
\end{equation}
where $T_p$ denotes the Chebyshev polynomial of the first kind of order $p$, considering
\begin{equation}\label{eq.beta.Gamma}
\theta_p=\frac{1}{\om_p}\sqrt{|Q(i\om_p)|^{2/p}-1}.
\end{equation}
the equilibrium point $E$ is asymptotically stable if $\theta\in(0,\theta_p)$. At $\theta=\theta_p$, system (\ref{sys.hpa.gr.dd}) undergoes a Hopf bifurcation at the equilibrium point $E$. 
\end{theorem}

\section{Numerical simulations}

The literature values of the elimination constants $w_i$, $i\in\{1,2,3\}$ are given by $w_i=\frac{\ln(2)}{T_i}$, where $T_i$ is the plasma half-life of hormones: $T_1\approx 4$ min, $T_2\approx 19.9$ min, $T_3\approx 76.4$ min \cite{Vinther_2011,Carroll_2007}. We choose $w_4=0.001~\text{min}^{-1}$ as in \cite{Sriram_2012}.

For simplicity, let $\eta=\mu=1$ and hence, the considered feedback functions are: 
$$f_1(x)=\frac{c_1^\alpha}{c_1^\alpha+x^\alpha}\quad,\quad f_2(x)=\frac{c_2^\alpha}{c_2^\alpha+x^\alpha}\quad,\quad f_3(x)=\frac{x^\beta}{c_3^\beta+x^\beta}$$
with $\alpha=4$ and $\beta=5$ as in \cite{Sriram_2012}, $c_1=2$ ng/ml as in \cite{Kaslik_Neamtu_2016} and $c_2=c_3=0.8$ ng/ml.

The normal equilibrium state $E$ should reflect the normal mean values of the hormones: $\bar{x}^n_1=7.659$ pg/ml (24-h mean value of CRH), $\bar{x}^n_2=21$ pg/ml (24-h mean value of ACTH) and $\bar{x}^n_3=3.055$ ng/ml (24-h mean value of free CORT) \cite{Carroll_2007}. In accordance with \cite{Gupta_2007}, we assume $\bar{x}^n_4=0.1$. Choosing $\xi=0.1$, from system (\ref{sys.eq.points}) we deduce:
\begin{align*}
k_1&=w_1 \df{\bar{x}^n_1}{f_1(\bar{x}^n_3)}=8.55261 \frac{\textrm{pg}}{\textrm{ml}\cdot\textrm{min}};\\
k_2&=w_2\df{\bar{x}^n_2}{\bar{x}^n_1f_2(\bar{x}^n_3\bar{x}^n_4)}=0.09753\textrm{ min}^{-1};\\
k_3&=w_3\frac{\bar{x}_3^n}{\bar{x}^n_2}= 1.31985\textrm{ min}^{-1};\\
k_4&=w_4\frac{\bar{x}^n_4}{\xi+f_3(\bar{x}^n_3\bar{x}_4)}=0.00092545\textrm{ min}^{-1}.
\end{align*}

For these values of the system parameters, the following equilibrium states exist:
\begin{align*}
& E^n=(7.659\textrm{ pg/ml}, 21 \textrm{ pg/ml}, 3.055 \textrm{ ng/ml},0.1) &  \textrm{normal state}\\
& E^d=(38.425\textrm{ pg/ml}, 10.04 \textrm{ pg/ml}, 1.4606 \textrm{ ng/ml},0.967) &  \textrm{diseased state}\\
& E^u=(8.3097\textrm{ pg/ml}, 20.495 \textrm{ pg/ml}, 2.981 \textrm{ ng/ml},0.16) &  \textrm{unstable state}
\end{align*}
The low level of cortisol in the case of the equilibrium state $E^d$ can be associated with hypocortisolism, and hence, $E^d$ is regarded as the "diseased" state. In the non-delayed case, the normal equilibrium state $E^n$ and the diseased equilibrium state $E^d$ are both asymptotically stable, as inequalities $(I_0)$, $(I_1)$ and $(I_2)$ are satisfied (see Theorem \ref{thm.stab}). On the hand, the equilibrium state $E^u$ is unstable, therefore, it is not significant from the biological point of view.  

It is important to emphasize that for both equilibria $E^n$ and $E^d$, inequality $\overline{(I_3)}$ is satisfied, which implies that when delays are introduced in the mathematical model, for sufficiently high average time delays bifurcations will occur, causing the loss of stability the $E^n$ and $E^d$.

As for the choice of mean time delays, firstly, as CRH travels from the hypothalamus to the pituitary through the hypophyseal portal blood vessels in an extremely short time \cite{Bairagi_2008}, we assume $\tau_1=0$. Moreover, the human inhibitory time course for the negative feedback of cortisol on the secretion of ACTH has been described as anything between 15 and 60 min \cite{Boscaro_1998,Posener_1997}, therefore we consider a mean delay $\tau_{32}\in(0,60]$. In our numerical simulations, we additionally assume that $\tau_{31}=\tau_{32}=\tau_{34}$. In \cite{Hermus_1984}, a 30-min delay has been given for the positive-feedforward effect of ACTH on plasma cortisol levels, therefore, we assume $\tau_2\in(0,30]$.

\subsection{Dirac kernels}

In the case of discrete time delays, choosing the bifurcation parameter $\tau=\tau_2+\tau_{32}$, we find the following critical values corresponding to Hopf bifurcations, based on Theorem \ref{thm.bif.dirac} and equation (\ref{eq.tau.Dirac}): $\tau_0^n=49.8505$ (min) for $E^n$ and $\tau_0^d=37.8362$ (min) for $E^d$, respectively. For $\tau<\tau_0^d$, both equilibria $E^n$ and $E^d$ are asymptotically stable. When $\tau $ crosses the critical value $\tau_0^d$, a Hopf bifurcation occurs in a neighborhood of the equilibrium $E^d$, which causes this equilibrium to become unstable and generates an asymptotically stable limit cycle in its neighborhood. The equilibrium state $E^n$ remains asymptotically stable whenever $\tau<\tau_0^n$. However, when the bifurcation parameter $\tau$ passes through the critical value $\tau_0^n$, a supercritical Hopf bifurcation takes place at $E^n$. Numerical simulations show that for $\tau>\tau_0^n$ two asymptotically stable limit cycles coexist, one corresponding to the normal ultradian rythm of the HPA axis and the other one reflecting a diseased hypocortisolic ultradian rythm. Considering $\tau=50$ (min), the coexisting limit cycles are presented in Figures \ref{fig:d1}, \ref{fig:d2} and \ref{fig:d3}.  

\begin{figure}[htbp]
\centering
\includegraphics[width=0.6\textwidth]{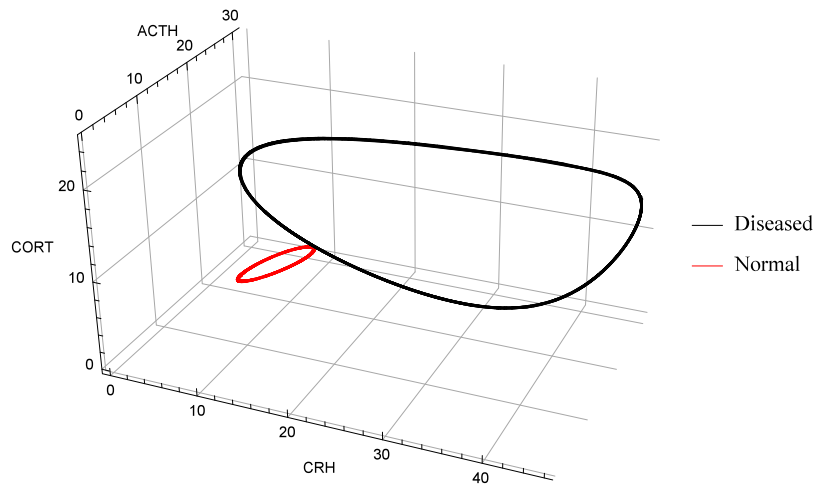} \caption{Two asymptotically stable limit cycles coexist in (\ref{sys.hpa.gr.dd}) with discrete delays: $\tau_1=0$, $\tau_2=30$ (min), $\tau_{31}=\tau_{32}=\tau_{34}=20$ (min).}
\label{fig:d1}
\end{figure}

\begin{figure}[htbp]
\centering
\includegraphics[width=0.7\textwidth]{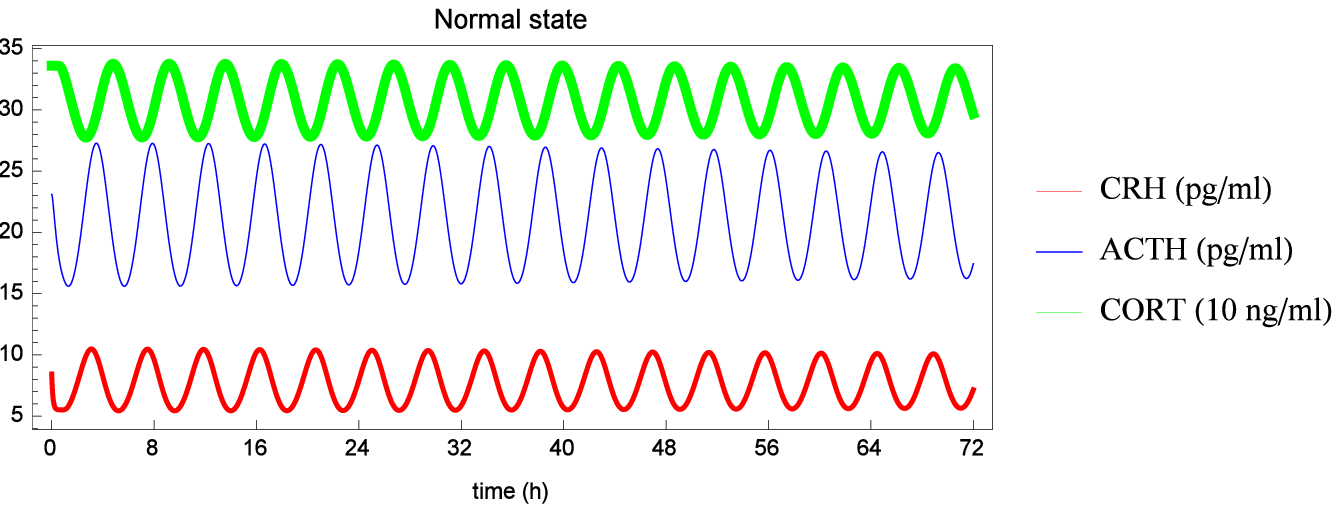} \caption{Evolution of the state variables of (\ref{sys.hpa.gr.dd}) with discrete delays: $\tau_1=0$, $\tau_2=30$ (min), $\tau_{31}=\tau_{32}=\tau_{34}=20$ (min) and an initial condition in a neighborhood of $E^n$.}
\label{fig:d2}
\end{figure}

\begin{figure}[htbp]
\centering
\includegraphics[width=0.7\textwidth]{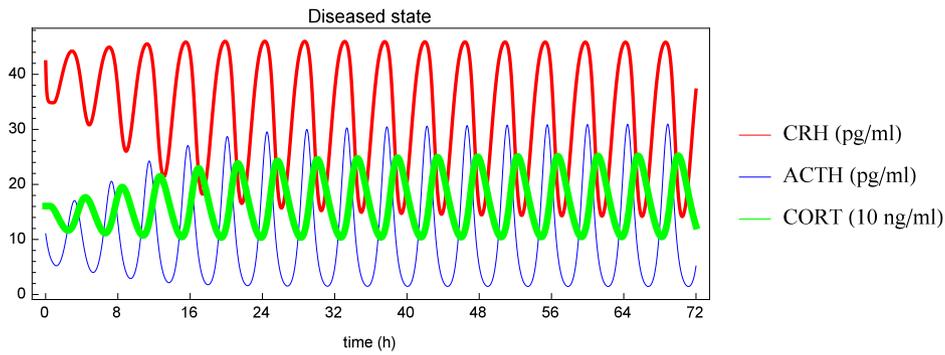} \caption{Evolution of the state variables of (\ref{sys.hpa.gr.dd}) with discrete delays: $\tau_1=0$, $\tau_2=30$ (min), $\tau_{31}=\tau_{32}=\tau_{34}=20$ (min) and an initial condition in a neighborhood of $E^d$.}
\label{fig:d3}
\end{figure}

\subsection{Strong Gamma kernels}

We now consider system  (\ref{sys.hpa.gr.dd}) with strong Gamma kernels with the same parameter $\theta$ and $p_2=p_{31}=p_{32}=p_{34}=2$ and $p_1=0$.  
Choosing the bifurcation parameter $\theta$, we find the following critical values corresponding to Hopf bifurcations, based on Theorem \ref{thm.bif.gamma} and equation (\ref{eq.beta.Gamma}): $\theta_4^d=12.625$ (min) for $E^d$ and $\theta_4^n=18.9$ (min) for $E^n$, respectively. As in the previous case, when $\theta$ passes one of the critical values $\theta_4^d$ or $\theta_4^n$, a supercritical Hopf bifurcation takes place in a neighborhood of the corresponding equilibrium $E^d$ or $E^n$. For $\theta>\theta_4^n$, numerical simulations show the coexistence of two asymptotically stable limit cycles, one corresponding to the normal ultradian rythm of the HPA axis and the other one reflecting a diseased hypocortisolic ultradian rythm. Considering $\theta=19$ (min) (i.e. a total average time delay $\tau=76$ (min)), the coexisting limit cycles are presented in Figures \ref{fig:g1}, \ref{fig:g2} and \ref{fig:g3}.  

\begin{figure}[htbp]
\centering
\includegraphics[width=0.6\textwidth]{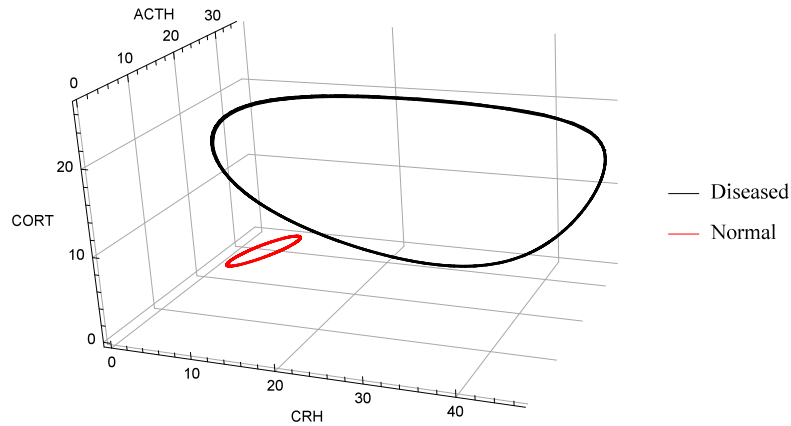} \caption{Two asymptotically stable limit cycles coexist in (\ref{sys.hpa.gr.dd}) with strong gamma kernels ($p_2=p_{31}=p_{32}=p_{34}=2$, $p_1=0$) with mean time delay $\theta=19$ (min).}
\label{fig:g1}
\end{figure}

\begin{figure}[htbp]
\centering
\includegraphics[width=0.7\textwidth]{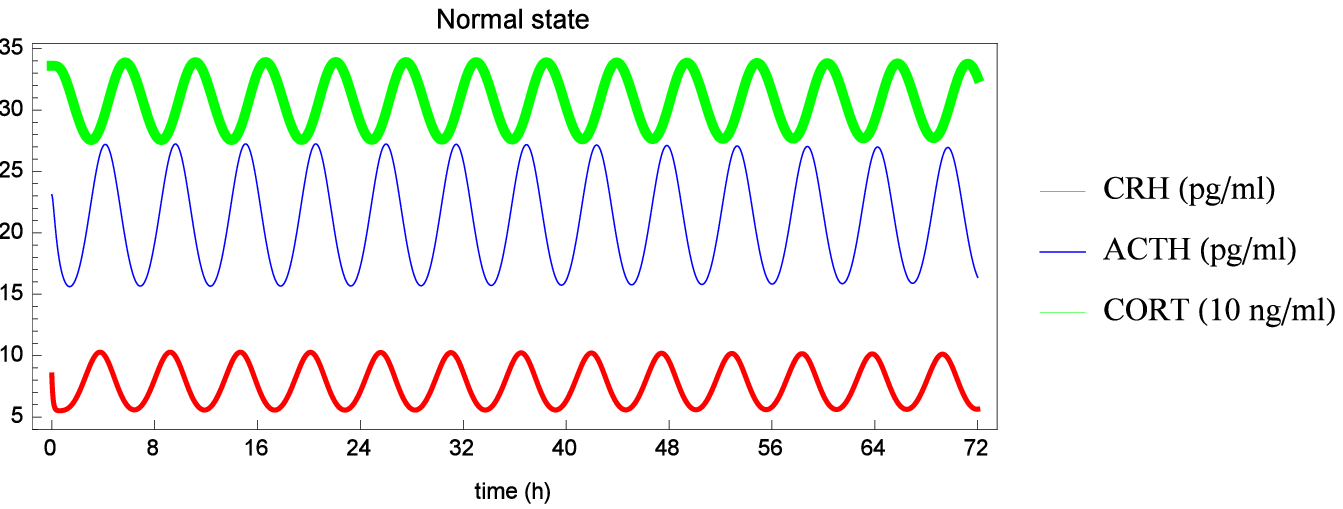} \caption{Evolution of the state variables of (\ref{sys.hpa.gr.dd}) with strong gamma kernels ($p_2=p_{31}=p_{32}=p_{34}=2$, $p_1=0$) with mean time delay $\theta=19$ (min) and an initial condition in a neighborhood of $E^n$.}
\label{fig:g2}
\end{figure}

\begin{figure}[htbp]
\centering
\includegraphics[width=0.7\textwidth]{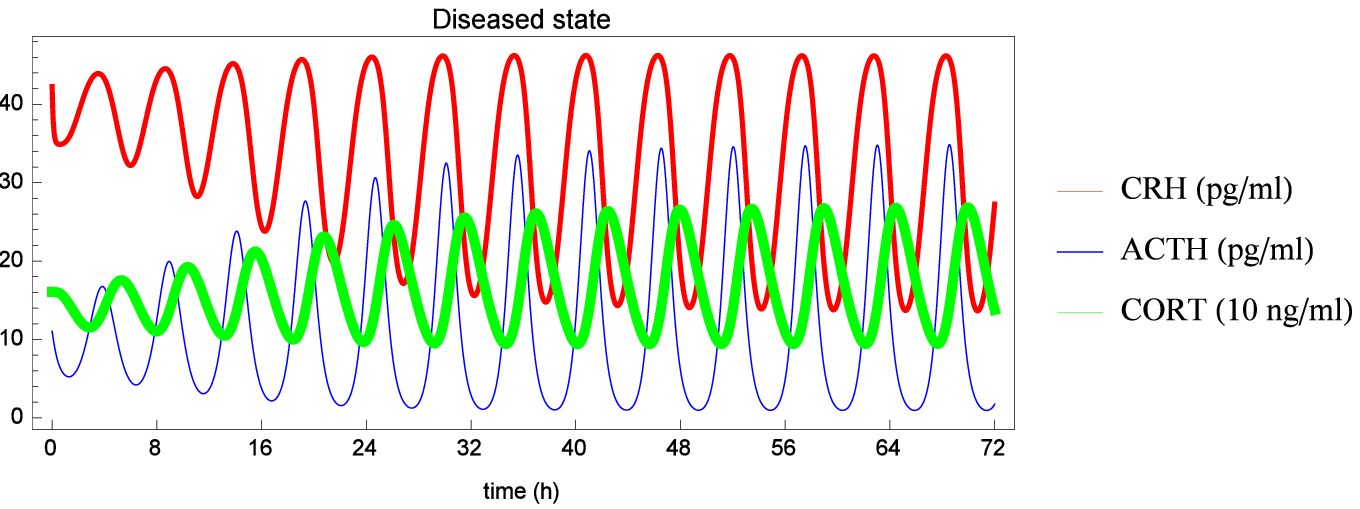} \caption{Evolution of the state variables of (\ref{sys.hpa.gr.dd}) with strong gamma kernels ($p_2=p_{31}=p_{32}=p_{34}=2$, $p_1=0$) with mean time delay $\theta=19$ (min) and an initial condition in a neighborhood of $E^d$.}
\label{fig:g3}
\end{figure}

\section{Conclusions}
This paper presents an analysis of a four-dimensional mathematical model describing the hypothalamus-pituitary-adrenal axis with the influence of the GR concentration, considering general feedback functions (which include as a special case Hill-type functions frequently used in the literature) to account for the interactions within the HPA axis.  Due to the fact that the involved processes are not instantaneous, general distributed delays have been included. This is a more realistic approach to the modeling of the biological processes, as it takes into account the whole past history of the variables, efficiently capturing the vital mechanisms of the HPA system. 

The positivity of the solutions and the existence of a positively invariant bounded region are proved. It is shown that the considered four-dimensional system has at least one equilibrium state and a detailed local stability and Hopf bifurcation analysis is given. Sufficient conditions expressed in terms of inequalities involving the system's parameters are found which guarantee the local asymptotic stability of an equilibrium. On the other hand, a necessary condition has also been obtained for the occurrence of bifurcations in a neighborhood of an equilibrium, when time delays are present. For the Hopf bifurcation analysis, two particular types of delays have been considered, given by Dirac and Gamma kernels, respectively.

Numerical simulations reflect the importance of the theoretical results. They exemplify the fact that an appropriate choice of the system's parameters leads to the coexistence of two asymptotically stable equilibria in the non-delayed case. When the total average time delay of the system passes through critical values which are computed according to the theoretical findings, the asymptotically stable equilibria loose their stability due to Hopf bifurcations and stable limit cycles are born in their neighborhoods. The coexistence of two stable limit cycles is revealed for a sufficiently large average time delay, which successfully model the ultradian rhythm of the HPA axis both in a normal disease-free situation and in a diseased hypocortisolim state, respectively. 

As a direction for future research, a fractional-order formulation of the mathematical model will be analyzed.

\bibliography{fractional_bib,hpa_bib}   

\end{document}